\documentclass[10pt]{article}  

\usepackage{amsmath, amsthm, amsfonts, amssymb}
\usepackage[bookmarks]{hyperref}
\usepackage{indentfirst} 
\usepackage{marvosym}
\usepackage{enumitem}
\usepackage[a4paper]{geometry}

\newtheorem{lemma}{Lemma}[section]
\newtheorem{theorem}[lemma]{Theorem}
\newtheorem{proposition}[lemma]{Proposition}
\newtheorem{corollary}[lemma]{Corollary}
\renewenvironment{proof}[1][\proofname]{{\noindent\bf #1. }}{\qed}
\newtheorem{theoremletters}{Theorem}

\newtheorem{corollaryletters}[theoremletters]{Corollary}

\theoremstyle{definition}
\newtheorem{remark}[lemma]{Remark}
\newtheorem{example}[lemma]{Example}
{\bf}{\rm}

\newcommand{\abs}[1]{\ensuremath{|#1|}}
\newcommand{\op}{\operatorname}
\newcommand{\ce}[2]{\pmb{\op{C}}_{#1}(#2)}

\newcommand{\ze}[1]{\pmb{\op{Z}}(#1)}
\newcommand{\hyp}[1]{\pmb{\op{Z}}_{\infty}(#1)}

\newcommand{\rad}[2]{\pmb{\op{O}}_{#1}(#2)}
\newcommand{\syl}[2]{\op{Syl}_{#1}\left(#2\right)}
\newcommand{\hall}[2]{\op{Hall}_{#1}\left(#2\right)}

\begin{document}

\title{\bf Determining hypercentral Hall subgroups in finite groups}

\author{\sc V. Sotomayor\thanks{Instituto Universitario de Matemática Pura y Aplicada, Universitat Politècnica de València, Camino de Vera s/n, 46022 Valencia, Spain. \newline
\Letter: \texttt{vsotomayor@mat.upv.es} \newline 
ORCID: 0000-0001-8649-5742 \newline \rule{6cm}{0.1mm}\newline
This research is supported by Ayuda a Primeros Proyectos de Investigación (PAID-06-23), Vicerrectorado de Investigación de la Universitat Politècnica de València (UPV), and by Proyecto CIAICO/2021/163 from Generalitat Valenciana (Spain). \newline
}}

\date{\textit{\small Dedicated to the memory of Francesco de Giovanni and Avinoam Mann}}

\maketitle

\begin{abstract}
\noindent Let $G$ be a finite group, and let $\pi$ be a set of primes. The aim of this paper is to obtain some results concerning how much information about the $\pi$-structure of $G$ can be gathered from the knowledge of the sizes of conjugacy classes of its $\pi$-elements and of their multiplicities. Among other results, we prove that this multiset of class sizes determines whether $G$ has a hypercentral Hall $\pi$-subgroup.

\medskip


\smallskip

\end{abstract}


\section{Introduction}

Within finite group theory, there are numerous results which endorse that the set of sizes of conjugacy classes of a group provide relevant information on its structure. However, there are some features that may not be inferred from this piece of information. For instance, the order of the group is unknown; less trivially, solubility cannot be determined as G. Navarro showed in \cite{N}, and the same happens with nilpotency as A.R. Camina and R.D. Camina established in \cite{CC}. In contrast, if the frequencies of the class sizes are also considered, then it is straightforward to compute the group order, and J. Cossey, T. Hawkes and A. Mann proved that nilpotency can be read off from that multiset (cf. \cite{CHM}). Nevertheless, whether it is possible to recognise solubility (or supersolubility) of a group from its class lenghts and their multiplicities is a problem that remains open nowadays, and it seems difficult to solve. It is worth mentioning that, as usual, the result due to Cossey, Hawkes and Mann has a counterpart in the context of irreducible character degrees, which was addressed by I.M. Isaacs in \cite{I}; in fact, it predates the conjugacy class version. We also refer the interested reader to \cite{Haw} and \cite{M} for further development concerning the multiset of character degrees.

To investigate whether the set of class sizes of some smaller subsets of elements of the group is enough for studying its algebraic structure is a research line that has received the interest of many authors over the last decades. Surprisingly enough, very little seems to have been done in the analogous context of either the multiset of class sizes or the multiset of character degrees. In this framework, A. Beltrán recently made some progress (cf. \cite{B}): he showed that the knowledge of the sizes of the conjugacy classes of a group $G$ that are contained in a normal subgroup $N$ and of their multiplicities determines if $N$ is hypercentral in $G$.

Inspired by the previous research, the aim of the present paper is to analyse whether some features of the $\pi$-structure of a group $G$ can be determined when only the class sizes of the $\pi$-elements of $G$ and their frequencies are considered, for a set of primes $\pi$. For this goal, we introduce, in a similar way to the papers cited above, the \emph{$\pi$-class size frequency function} of $G$, $w^G_{\pi}(n):\mathbb{N}\longrightarrow \mathbb{N}$, as follows: $$ w^G_{\pi}(n)=\frac{1}{n}\cdot|\{g\in G_{\pi}\; : \; |G:\ce{G}{g}|=n\}|,$$ where $G_{\pi}$ denotes the set of $\pi$-elements of $G$. In other words, for every natural number $n$, the function $w^G_{\pi}(n)$ computes the number of conjugacy classes of $\pi$-elements of $G$ that have size $n$. This function may provide some arithmetical properties of the $\pi$-structure of $G$, as the order of the Hall $\pi$-subgroup of its centre $\ze{G}$, i.e. the number of $\pi$-elements that are central in $G$; the $\pi$-part of the order of its hypercentre can also be computed (see Theorem \ref{theo-hyp}), and this is actually a key ingredient in our proofs. Nonetheless, other elementary facts, as the $\pi$-part of the order of $G$, cannot be deduced if $\pi$ is a proper subset of the set of prime divisors of the order of $G$: for example, the symmetric group on three letters and the alternating group on four letters, with $\pi=\{2\}$, have equal $\pi$-class size frequency function, although the $2$-parts of their orders differ.

This loss of information about class sizes of elements of $G$ with orders divisible by prime numbers lying in $\pi'$ entails the use of a different approach in some arguments. For instance, at some point a celebrated Frobenius' theorem regarding the number of solutions of $x^n=1$ in a group $G$ such that $n$ is a divisor of its order is involved.

Our first result shows that the $\pi$-class size frequency function of $G$ determines the existence of hypercentral Hall $\pi$-subgroups. We highlight that no $\pi$-separability assumption is required.

\begin{theoremletters}
\label{teoA}
Let $G$ and $H$ be finite groups, and suppose that $G$ has a hypercentral Hall $\pi$-subgroup for certain set of primes $\pi$. If the conjugacy class sizes of the $\pi$-elements of $G$ joint with their multiplicities are the same as those of $H$, then $H$ has a hypercentral Hall $\pi$-subgroup.
\end{theoremletters}

A particular case that merits to be emphasised is when $\pi=p'$, since it follows that the nilpotency of $G$ can be determined from its $p'$-class size frequency function. Remember that the main result of \cite{CHM} attained the same conclusion but considering all the class sizes of $G$, which is a larger multiset in general. Recall that a $p$\emph{-regular} element of $G$ is an element whose order is not divisible by $p$.

\begin{corollaryletters}
\label{corB}
Let $G$ and $H$ be finite groups, and suppose that $G$ is nilpotent. If, for a given prime $p$, the class sizes of the $p$-regular elements of $G$ and their multiplicities coincide with those of $H$, then $H$ is nilpotent.
\end{corollaryletters}

It is noteworthy that the information provided by the $\pi$-class size frequency function is sometimes quite restrictive, and thus to generalise known results that consider the whole multiset of class lengths is not always possible. In Remark \ref{example_mattarei} we will further discuss this feature.

At this point, it is convenient to point out that the positive results in the literature on this topic usually give a concrete criterion to read off the desired property from the multiset of class sizes (or character degrees). In our situation, however, we have not been able to achieve it in all cases because, as previously said, even the $\pi$-part of the order of the group is unknown. Exceptionally, and following the spirit of \cite{M}, we give the next sufficient condition to detect when the only $q$-elements with class length not divisible by $q$ are the central ones, where $q\in \pi$. Let us denote by $\mathcal{S}_{q'}(G_{\pi})$ the union of those conjugacy classes of $\pi$-elements of $G$ whose cardinalities are not divisible by the prime $q$.

\begin{theoremletters}
\label{teoC}
Let $G$ be finite group, let $\pi$ be a set of primes and let $q\in\pi$. Then $|\ze{G}|_q$ divides $|\mathcal{S}_{q'}(G_{\pi})|_q$. Further, if $|\ze{G}|_q=|\mathcal{S}_{q'}(G_{\pi})|_q$, then $\ze{Q}\leqslant\ze{G}$ where $Q$ is a Sylow $q$-subgroup $G$.
\end{theoremletters}

The second assertion in the above theorem is simply not true when $q\notin \pi$, it is enough to consider as $G$ the symmetric group on three letters, with $\pi=\{3\}$ and $q=2$. We also demonstrate that the converse holds whenever $\ce{G}{Q}$ has a normal Hall $\pi$-subgroup, and in addition we illustrate with an example that this condition is not necessary (see Proposition \ref{propC} and Example \ref{example_prop}).

In the previous results, evidence has been shown on how the $\pi$-class size frequency function gives information about properties of a Hall $\pi$-subgroup with regard to its immersion in the group. It is then natural to wonder whether this function may also provide properties of the Hall $\pi$-subgroup itself. To determine solubility or superso\-lubility, as occurs when $\pi$ is the set of prime divisors of the group order, are likewise open pro\-blems. The author has also been unable to decide if nilpotency is recognisable. However, abelianity is almost determined, up to some specific cases, and via the classification of finite simple groups; this contrasts with the situation when $\pi$ is the whole set of prime divisors of the group order, where abelianity is elementarily recognisable. We will further comment this issue in Remark \ref{remark_open}.


\section{Notation and preliminaries}

Hereafter, if $x$ is an element of a finite group $G$, then we denote by $x^G$ the conjugacy class of $x$ in $G$, and its size is $\abs{x^G}=\abs{G:\ce{G}{x}}$. We write $\pi(G)$ for the set of prime divisors of $\abs{G}$. For a positive integer $n$ and a prime number $p$, the $p$-part $n_{p}$ of $n$ is the largest power of $p$ that divides $n$. In particular, if $\pi$ is a set of primes, then the $\pi$-part $n_{\pi}$ of $n$ is the product of $n_p$ for each prime $p\in\pi$. If $n_{\pi'}=1$, then $n$ is said to be a $\pi$-number. As usual, the set of all Sylow $p$-subgroups of $G$ is denoted by $\syl{p}{G}$, and $\hall{\pi}{G}$ is the set of all Hall $\pi$-subgroups of $G$. A group is called $\pi$-decomposable if it has a Hall $\pi$-subgroup as a direct factor. We recall that the hypercentre $\hyp{G}$ of a group $G$ is the last term of its upper central series, i.e. the last term of the series $$1\leqslant\pmb{\text{Z}}_1(G)\leqslant \pmb{\text{Z}}_2(G)\leqslant\cdots,$$ where $\pmb{\text{Z}}_1(G)=\ze{G}$ and $\pmb{\text{Z}}_{i+1}(G)$ is defined by $\pmb{\text{Z}}_{i+1}(G)/\pmb{\text{Z}}_i(G)=\ze{G/\pmb{\text{Z}}_i(G)}$ for every integer $i\geq 1$. It is well-known that that $\hyp{G}$ is nilpotent, it is a characteristic subgroup of $G$, and $\hyp{G}=G$ if and only if $G$ is nilpotent (cf. \cite{H3}). Finally, and as mentioned in the Introduction, we denote by $G_{\pi}$ the set of $\pi$-elements of $G$; and if $\sigma$ is another set of primes, $\mathcal{S}_{\sigma}(G_{\pi})$ is the union of those conjugacy classes $g^G$ with $g\in G_{\pi}$ such that $|g^G|$ is divisible by primes in $\sigma$ only; thus its size is $$ |\mathcal{S}_{\sigma}(G_{\pi})| = \displaystyle\sum_{n_{\sigma'}=1} w^G_{\pi}(n)\cdot n,$$ where $w^G_{\pi}$ is the $\pi$-class size frequency function previously defined. In particular, we will write $\mathcal{S}_q(G_{\pi})$ when $\sigma$ consists of a single prime $q$, and $\mathcal{S}_{q'}(G_{\pi})$ when $\sigma=\{q\}'$. The remaining notation and terminology used is standard in the framework of finite group theory.

Let us state two preliminary results that will be needed later. The first one is an elementary lemma, whilst the second one is a celebrated theorem due to Frobenius.

\begin{lemma}
\label{direct-prod}
Let the finite group $G=A\times B$ be the direct product of two subgroups $A$ and $B$. Then the multiset of conjugacy class sizes in $G$ of the elements lying in $A$ coincides with the multiset of class sizes of $A$.
\end{lemma}

\begin{proof}
This follows immediately from the fact that every element $x\in A$ satisfies $x^G=x^A$.
\end{proof}

\medskip

\begin{theorem}
\label{frobenius}
If $n$ is a divisor of $|G|$ for a finite group $G$, then the number of solutions of $x^n=1$ in $G$ is a multiple of $n$.
\end{theorem}

\begin{proof}
See \cite[9.9 Theorem (b)]{H3} for a proof based on character theory, or \cite[Theorem 9.1.2]{Hall} for a group-theoretic proof.
\end{proof}


\section{Proof of main results}

We point out that the proof of Theorem \ref{teoA} differs from that originally presented in \cite{CHM}; instead, we have followed the ideas in \cite[Theorem 23.5]{H3} with suitable changes.

\begin{theorem}
\label{theo-hyp}
Let $G$ be a group, and $\pi$ a set of primes. Then for every $q\in\pi$ it holds $$|\mathcal{S}_q(G_{\pi})|_q=|\hyp{G}|_q.$$ In particular, a Sylow $q$-subgroup of $G$ is hypercentral if and only if $|G|_q=|\mathcal{S}_q(G_{\pi})|_q$.
\end{theorem}

\begin{proof}
We argue by induction on $|G|$. Since $\mathcal{S}_q(G_{\pi})$ is the set of $\pi$-elements of $G$ whose class sizes are powers of $q$, then it certainly holds that $|\mathcal{S}_q(G_{\pi})|\equiv |\ze{G}|_{\pi} \; (\text{mod } q)$. Thus, if $q$ does not divide $|\hyp{G}|$, then it cannot divide $|\ze{G}|_{\pi}$ either, so $|\mathcal{S}_q(G_{\pi})|_q=1=|\hyp{G}|_q$ as wanted.

Now we may suppose that $q$ is a prime divisor of $|\hyp{G}|$. It is well-known that then $q$ also divides $|\ze{G}|$ (see for instance \cite[Lemma 3]{P}), so we can take a normal subgroup $N\leqslant \ze{G}$ of order $q$. Let us write $\overline{G}=G/N$. We claim that \begin{eqnarray} & |\mathcal{S}_q(\overline{G}_{\pi})|\cdot q=|\mathcal{S}_q(G_{\pi})|. \label{eq} \end{eqnarray} Let $\overline{g}\in\overline{G}$, and denote by $\overline{C}=\ce{\overline{G}}{\overline{g}}$. Let $\alpha:C\longrightarrow N$ be the map defined by $\alpha(x)=[g,x]\in N$, for every $x\in C$. Since $N\leqslant\ze{G}$, observe that $\alpha(xy)=[g,xy]=[g,y][g,x]^y=[g,y][g,x]=[g,x][g,y]=\alpha(x)\alpha(y)$ for every $x,y\in C$, so $\alpha$ is a group homomorphism. As $\ker(\alpha)=\ce{G}{g}$, it holds that $|C/\ce{G}{g}|$ divides $|N|=q$. 

Take $\overline{g}\in\mathcal{S}_q(\overline{G}_{\pi})$, i.e. a $\pi$-element $\overline{g}\in \overline{G}$ such that $|\overline{g}^{\overline{G}}|=q^n$ for some integer $n\geq 0$. We may clearly assume that $g$ is a $\pi$-element. Set $N=\langle z\rangle$. As $|C/\ce{G}{g}|$ divides $|N|=q$, then either $C=\ce{G}{g}$ or $|C|=q\cdot|\ce{G}{g}|$. In the former case $\overline{C}=\overline{\ce{G}{g}}$, and since $N\leqslant \ze{G}$ and $q\in \pi$, then $\{g, gz, ..., gz^{q-1}\}$ are $q$ distinct $\pi$-elements of $G$, and it is not difficult to show that $\{g^G, (gz)^G, ..., (gz^{q-1})^G\}$ are $q$ distinct conjugacy classes that are preimages of $\overline{g}^{\overline{G}}$. Moreover, they all have the same length because $|(gz^i)^G|=|(g^G)z^i|=|g^G|$ for every $0\leq i \leq q-1$, and this length is precisely $$|g^G|=|G:\ce{G}{g}|=|\overline{G}:\overline{C}|=|\overline{g}^{\overline{G}}|.$$ On the other hand, in the latter case $|C|=q\cdot|\ce{G}{g}|$, so $$|g^G|=|G:\ce{G}{g}|=q\cdot |G:C|=q\cdot |\overline{G}:\overline{C}|=q\cdot |\overline{g}^{\overline{G}}|,$$ and one can check that the unique preimage of $\overline{g}^{\overline{G}}$ in this case is $g^G$. Consequently, each $\overline{g}^{\overline{G}}$, where $\overline{g}\in\mathcal{S}_q(\overline{G}_{\pi})$, yields either $q$ conjugacy classes of $\pi$-elements of $G$ with the same size as $|\overline{g}^{\overline{G}}|$, or one conjugacy class of a $\pi$-element of $G$ with size exactly $q\cdot |\overline{g}^{\overline{G}}|$. Hence, for every integer $i\geq 0$, we can decompose $w^{\overline{G}}_{\pi}(q^i)=\beta^{\overline{G}}_{\pi}(q^i) + \gamma^{\overline{G}}_{\pi}(q^i)$, where $\beta^{\overline{G}}_{\pi}(q^i)$ and $\gamma^{\overline{G}}_{\pi}(q^i)$ denote the number of conjugacy classes of $\pi$-elements of $\overline{G}$ that correspond to each of the mentioned type of classes. Thus \begin{eqnarray*}
|\mathcal{S}_q(G_{\pi})| & =& \displaystyle\sum_{i=0}^{\infty} q^i\cdot q\cdot \beta^{\overline{G}}_{\pi}(q^i) + \sum_{i=0}^{\infty} q^{i+1}\cdot \gamma^{\overline{G}}_{\pi}(q^i) \\ & = & \sum_{i=0}^{\infty} q^{i+1}\cdot (\beta^{\overline{G}}_{\pi}(q^i)+\gamma^{\overline{G}}_{\pi}(q^i)) \\
& = & q\cdot \sum_{i=0}^{\infty} q^{i}\cdot w^{\overline{G}}_{\pi}(q^i) \\
& = & q\cdot |\mathcal{S}_q(\overline{G}_{\pi})|,
\end{eqnarray*} so equation equation (\ref{eq}) is established. Finally, since $N$ is central in $G$, then $\hyp{\overline{G}}=\overline{\hyp{G}}$ by the definition of the hypercentre. Now using induction and equation (\ref{eq}) we obtain $$|\mathcal{S}_q(G_{\pi})|_q=q\cdot |\mathcal{S}_q(\overline{G}_{\pi})|_q=q\cdot |\hyp{\overline{G}}|_q=q\cdot|\overline{\hyp{G}}|_q=|\hyp{G}|_q,$$ as desired. The last assertion of the theorem is straightforward. 
\end{proof}

\begin{remark} 
(1) The above theorem is simply not true whenever $q\notin \pi$. It is enough to consider any $q$-group $G$, which has a unique conjugacy class of $\pi$-elements (the trivial one) with cardinality a $q$-power, for any set of primes $\pi$ that does not contain $q$, so $|\mathcal{S}_q(G_{\pi})|=1\neq |G|$.

\noindent (2) We have previously mentioned that from the $\pi$-class size frequency function of $G$ it is possible to retrieve both $|\ze{G}|_{\pi}$ and $|\hyp{G}|_{\pi}$. However, we point out that the $\pi$-part $|\pmb{\text{Z}}_i(G)|_{\pi}$ for the intermediate terms of the upper central series of $G$ cannot be computed, even when $\pi=\pi(G)$, as it was shown in \cite[Example 1]{CHM}.

\noindent (3) If we take $\pi=\{q\}$ for a given prime $q$, then Theorem \ref{theo-hyp} yields a $q$-decomposability criterion for a group based on the multiset of class lengths of its $q$-elements only.

\noindent (4) Theorem \ref{theo-hyp} asserts that a group is $q$-decomposable if and only if $|\mathcal{S}_q(G_{\pi})|_q$ is as large as possible, for any set of primes $\pi$ that contains $q$.
\end{remark}

The result below is \cite[Theorem]{CHM}. For the sake of completeness, we also give the proof.

\begin{corollary}
Let $G$ and $H$ be finite groups. Suppose that $G$ is nilpotent, and that its conjugacy class sizes joint with their multiplicities are the same as those of $H$. Then $H$ is nilpotent.
\end{corollary}

\begin{proof}
The hypotheses clearly lead to $|G|=|H|$, because both numbers can be computed from the class sizes and their multiplicities. Since $G$ is nilpotent, then every Sylow subgroup of $G$ is hypercentral, so applying Theorem \ref{theo-hyp} with $\pi=\pi(G)$ it follows for every prime $q\in\pi$ that $|\mathcal{S}_q(G)|_q=|\hyp{G}|_q=|G|_q$. Therefore we get $$ |\mathcal{S}_q(H)|_q=|\mathcal{S}_q(G)|_q=|G|_q=|H|_q,$$ where the first equality is due to our assumptions. Thus, by Theorem \ref{theo-hyp} again, we can affirm that $H$ has hypercentral Sylow subgroups for every prime $q\in\pi(G)=\pi(H)$, so $H$ is nilpotent.
\end{proof}

\medskip

\begin{proof}[Proof of Theorem \ref{teoA}]
Let us suppose that $G$ and $H$ have the same class sizes of $\pi$-elements, taking into account their multiplicities, and that $G$ has a hypercentral Hall $\pi$-subgroup. Hence $G=\rad{\pi}{G}\times\rad{\pi'}{G}$ with $\rad{\pi}{G}$ nilpotent, and by Lemma \ref{direct-prod} the multiset of class sizes in $G$ of its $\pi$-elements is actually the whole multiset of class sizes of $\rad{\pi}{G}$. In particular we can compute $|G|_{\pi}=|\rad{\pi}{G}|$.

Since the class sizes of $\pi$-elements of $H$ and their multiplicities are the same as those of $G$, and they add up to $|G|_{\pi}$, then applying Theorem \ref{frobenius} to $H$ with $n=|H|_{\pi}$ we get that the number of solutions of $x^{|H|_{\pi}}=1$ in $H$ is a multiple of $|H|_{\pi}$. In other words, $|H|_{\pi}$ divides the number of $\pi$-elements of $H$, and we have that this last number equals $|G|_{\pi}$. On the other hand, for every prime $q\in \pi$, we have $$|G|_q=|\hyp{G}|_q=|\mathcal{S}_q(G_{\pi})|_q=|\mathcal{S}_q(H_{\pi})|_q=|\hyp{H}|_q,$$ where the first equality follows from the fact that $G=\rad{\pi}{G}\times\rad{\pi'}{G}$ with $\rad{\pi}{G}$ nilpotent, the second and fourth ones are due to Theorem \ref{theo-hyp}, and the third one is due to our assumptions. It follows that $|G|_q=|\hyp{H}|_q$ divides $|H|_q$ for every $q\in\pi$, and since we have also proved that $|H|_{\pi}$ divides $|G|_{\pi}$, then we get $|G|_q=|\hyp{H}|_q=|H|_q$ for every $q\in\pi$. Thus $H$ has a hypercentral Hall $\pi$-subgroup.
\end{proof}

\begin{remark}
\label{example_mattarei}
As discussed earlier, to generalise known results that consider the whole multiset of class lengths is not always possible. For instance, S. Mattarei proved that $G$ has a Sylow $q$-subgroup as a direct factor if and only if the number of elements with class lengths not divisible by $q$ is exactly $|G|_{q'}\cdot |\ze{G}|_q$, where $q$ is a prime number (cf. \cite[Theorem 3.3]{M}). Observe that, if $G=\rad{q}{G}\times \rad{q'}{G}$, then in the most extreme case with $\pi=\{q\}$ it holds that $|\mathcal{S}_{q'}(G_{\pi})|=|\ze{G}|_q$, so it would be natural to wonder whether \emph{a Sylow $q$-subgroup is a direct factor of $G$ if and only if the number of $q$-elements with class lengths not divisible by $q$ is exactly $|\ze{G}|_q$}. But this is certainly not true, since there are two $2$-elements in the group $SL(2,3)$ with class sizes not divisible by $2$, the order of $\ze{G}$ is $2$, and the Sylow $2$-subgroup is not a direct factor.
\end{remark}

\begin{proof}[Proof of Theorem \ref{teoC}]
Let $q\in \pi$ be a prime, and $Q\in\syl{q}{G}$. We aim to prove that $|\ze{G}|_q$ divides $|\mathcal{S}_{q'}(G_{\pi})|_q$, and if equality holds, then $\ze{Q}\leqslant\ze{G}$. Set $Z:=\ze{G}\cap Q$, so $Z\in\syl{q}{\ze{G}}$ and it is a normal subgroup of $G$. Notice that $\mathcal{S}_{q'}(G_{\pi})$ can be written as a union of cosets of $Z$: this follows from the fact that, if $g\in\mathcal{S}_{q'}(G_{\pi})$, since $Z$ is a central $q$-group with $q\in\pi$, then every $z\in Z$ satisfies that $gz\in G_{\pi}$ and $|(gz)^G|=|(g^G)z|=|g^G|$, so $gZ\subseteq \mathcal{S}_{q'}(G_{\pi})$. Thus the first assertion is established.

Henceforth we suppose that $|\ze{G}|_q=|\mathcal{S}_{q'}(G_{\pi})|_q$. We claim that if $g\in\mathcal{S}_{q'}(G_{\pi})$, then $\ce{G/Z}{gZ}\subseteq \ce{G}{g}/Z$, and since the other containment is trivial, then equality follows. Set $C/Z=\ce{G/Z}{gZ}$. Identically as in the proof of Theorem \ref{theo-hyp}, as $Z\leqslant\ze{G}$, it holds that the map $\alpha: C\longrightarrow Z$ defined by $\alpha(x)=[g,x]$ for every $x\in C$ is a group homomorphism, whose kernel is $\ce{G}{g}$. Therefore $|C/\ce{G}{g}|$ divides $|Z|$. But $Z$ is a $q$-group, and $|C/\ce{G}{g}|$ divides $|G:\ce{G}{g}|=|g^G|$ which is a $q'$-number because $g\in\mathcal{S}_{q'}(G_{\pi})$, so $C=\ce{G}{g}$ as wanted. 

Let us denote by $\mathcal{A}$ the set of cosets of $Z$ in $\mathcal{S}_{q'}(G_{\pi})$, and consider the action of $Q/Z$ by conjugation on $\mathcal{A}$. By the previous paragraph, we have that a coset $gZ$ is fixed by $Q/Z$ if and only if $Q$ fixes $g$, which happens if and only if $gZ$ is a $\pi$-element lying in $\ce{G}{Q}/Z$. Let us denote by $t$ the number of $\pi$-elements in $\ce{G}{Q}/Z$. The sizes of the orbits lead to $$|\mathcal{S}_{q'}(G_{\pi})|/|Z|=|\mathcal{A}| \equiv t \; (\text{mod } q).$$ Recall that $Z=\ze{G}\cap Q$. We deduce that if $|\ze{G}|_q=|\mathcal{S}_{q'}(G_{\pi})|_q$, then $t$ is not divisible by $q$. Since $t$ is the number of $\pi$-elements that lie in $\ce{G}{Q}/Z$, then $t$ is a multiple of $|\ce{G}{Q}/Z|_{\pi}$ by Theorem \ref{frobenius}. Thus, if $t$ is not divisible by $q$, then $|\ce{G}{Q}/Z|_{\pi}$ is not either, so $|\ce{G}{Q}|_q=|Z|$. This last feature occurs if and only if $Z=\ze{Q}$ because $\ze{Q}$ is the unique Sylow $q$-subgroup of $\ce{G}{Q}$. 
\end{proof}

\begin{remark}
Theorem \ref{teoC} admits the following restatement: if $|\mathcal{S}_{q'}(G_{\pi})|_q$ is as small as possible, then the unique $q$-elements with class size not divisible by $q$ are the central ones.
\end{remark}

\begin{example}
The converse of the second assertion in Theorem \ref{teoC} does not hold in general. To see this, let $G$ be the direct product of the cyclic group of order 3 and the dihedral group of order 10. Set $\pi=\{2,3\}$ and $q=3$. The centre of $G$ coincides with the unique Sylow $3$-subgroup $Q$ of $G$, so certainly $\ze{Q}\leqslant\ze{G}$. Nonetheless, $|\mathcal{S}_{3'}(G_{\pi})|_3=|G_{\pi}|_3=18_3=9\neq 3=|\ze{G}|_3$. 
\end{example}

We give below a sufficient condition for the converse of Theorem \ref{teoC} to be true.

\begin{proposition}
\label{propC}
Let $G$ be finite group, let $\pi$ be a set of primes. Let $Q\in\syl{q}{G}$ for a prime $q\in\pi$. If $\ze{Q}\leqslant\ze{G}$ and $\ce{G}{Q}$ has a normal Hall $\pi$-subgroup, then $|\ze{G}|_q=|\mathcal{S}_{q'}(G_{\pi})|_q$.
\end{proposition}

\begin{proof}
Set $Z:=Q\cap\ze{G}$. If $\ze{Q}\leqslant\ze{G}$, then $Z=\ze{Q}$, and $\ce{G}{Q}/Z$ is a $q'$-group. Similarly as in the proof of Theorem \ref{teoC}, let us consider the action of $Q/Z$ by conjugation on the set $\mathcal{A}$ of cosets of $Z$ in $\mathcal{S}_{q'}(G_{\pi})$. Hence $$|\mathcal{S}_{q'}(G_{\pi})|/|Z|=|\mathcal{A}|\equiv t\; (\text{mod } q),$$ where $t$ is the number of $\pi$-elements in $\ce{G}{Q}/Z$. Since $\ce{G}{Q}/Z$ has a unique Hall $\pi$-subgroup by assumptions, then it follows $t=|\ce{G}{Q}/Z|_{\pi}$. But we have seen that $\ce{G}{Q}/Z$ is a $q'$-group, and then $q$ does not divide $|\mathcal{S}_{q'}(G_{\pi})|/|Z|$, which leads to $|Z|_q=|\ze{G}|_q=|\mathcal{S}_{q'}(G_{\pi})|_q$.
\end{proof}

\begin{remark}
It is significant to mention that the previous condition regarding the existence of a normal Hall $\pi$-subgroup in $\ce{G}{Q}$ is trivially satisfied if $\pi=\{q\}$, since $\ze{Q}$ is its unique Sylow $q$-subgroup. Furthermore, if $\pi=\pi(G)$, then that assumption is also superfluous, and in particular we recover \cite[Theorem 3.1]{M}. 
\end{remark}

\begin{example}
\label{example_prop}
The assumption in Proposition \ref{propC} about the existence of a normal Hall $\pi$-subgroup in $\ce{G}{Q}$ is not necessary. To see this, consider as $G$ the direct product of the cyclic group of order $5$ and the symmetric group on three letters. Set $\pi=\{2,5\}$ and $q=5$. Then $\ce{G}{Q}=G$ does not have a normal Hall $\pi$-subgroup, the centre of the Sylow $5$-subgroup $Q$ of $G$ is central, and both $|\ze{G}|_5$ and $|\mathcal{S}_{5'}(G_{\pi})|_5=|G_{\pi}|_5$ are equal to $5$.
\end{example}

\begin{remark}
\label{remark_open}
If $G$ is a finite group, and $\pi=\pi(G)$, then the $\pi$-class size frequency function undoubtedly determines the abelianity of the group. But this problem cannot be elementary transferred to the case $\pi\subsetneq \pi(G)$. Let us suppose that $G$ and $H$ are two finite groups such that, for certain subset $\pi$, their multisets of class sizes of $\pi$-elements coincide, and $G$ has abelian Hall $\pi$-subgroups. Then certainly the class sizes of all $\pi$-elements of $G$, and so also those of $H$, are $\pi'$-numbers. By \cite[Theorem B]{BFM...}, which uses the classification of finite simple groups, it follows that $H$ has nilpotent Hall $\pi$-subgroups. Further, $H$ has abelian Hall $\pi$-subgroups whenever $\pi\cap \{3,5\}=\emptyset$ by \cite[Theorem C]{BFM...}. Finally, if $\pi\cap \{3,5\}\neq\emptyset$ and $H$ has no composition factor isomorphic to $Ru, J_4, Th$ or $^{2}F_4(q_i)'$ with $q_i+1$ not divisible by $9$, then by \cite[Theorem B]{N2} we deduce that the Hall $\pi$-subgroups of $H$ are also abelian. Thus the open question is whether it is possible that $H$ has a composition factor isomorphic to one of those types of simple groups, under the assumptions that $G$ has abelian Hall $\pi$-subgroups and both groups have the same multiset of class sizes of $\pi$-elements.
\end{remark}


\noindent \textbf{Acknowledgements.} The author would like to thank Antonio Beltrán for helpful discussions about the paper.


\end{document}